\documentclass[draft]{amsart}
\usepackage{amsmath}
\usepackage{amssymb}
\usepackage{graphicx}
\newtheorem{theorem}{Theorem}[section]
\newtheorem{lemma}[theorem]{Lemma}

\theoremstyle{definition}
\newtheorem{definition}[theorem]{Definition}

\newtheorem{remark}[theorem]{Remark}

\numberwithin{equation}{section}

\begin{document}

\title[The asymptotic quantization error for the doubling measures]{On the asymptotic quantization error for the doubling measures on Moran sets}
\author{Sanguo Zhu}
\address{School of Mathematics and Physics, Jiangsu University
of Technology\\ Changzhou 213001, China.}
\email{sgzhu@jsut.edu.cn}


\begin{abstract}
We study the quantization errors for the doubling probability measures $\mu$ which are supported on a class of Moran sets $E\subset\mathbb{R}^q$. For each $n\geq 1$, let $\alpha_n$ be an arbitrary $n$-optimal set for $\mu$ of order $r$ and $\{P_a(\alpha_n)\}_{a\in\alpha_n}$ an arbitrary Voronoi partition with respect to $\alpha_n$. We denote by $I_a(\alpha_n,\mu)$ the integral $\int_{P_a(\alpha_n)}d(x,a)^rd\mu(x)$ and define
\begin{eqnarray*}
\underline{J}(\alpha_n,\mu):=\min\limits_{a\in\alpha_n}I_a(\alpha_n,\mu),\; \overline{J}(\alpha_n,\mu):=\max\limits_{a\in\alpha_n}I_a(\alpha_n,\mu).
\end{eqnarray*}
Let $e_{n,r}(\mu)$ denote the $n$th quantization error for $\mu$ of order $r$. Assuming a version of the open set condition for $E$, we prove that
\[
\underline{J}(\alpha_n,\mu),\overline{J}(\alpha_n,\mu)\asymp\frac{1}{n}e_{n,r}^r(\mu).
\]
This result shows that, for the doubling measures on Moran sets $E$, a weak version of Gersho's conjecture holds.
\end{abstract}

\maketitle

\section{Introduction}

One of the main objectives of the quantization problem is to study the error in the approximation of a given probability measures with discrete measures of finite support. We refer to \cite{GN:98} for the deep background of this problem and \cite{GL:00,GL:04} for rigorous mathematical foundations of quantization theory.

For each $n\geq 1$, we write $\mathcal{D}_n:=\{\alpha\subset\mathbb{R}^q:{\rm card}(\alpha)=n\}$. Let $\nu$ be a Borel probability measure on $\mathbb{R}^q$. Let $d$ denote the metric induced by an arbitrary norm on $\mathbb{R}^q$ (in the following, we work with the Euclidean norm). The $n$th quantization error for $\mu$ of order $r\in(0,\infty)$ can be defined by
\begin{eqnarray}\label{error}
e_{n,r}(\nu):=\bigg(\inf_{\alpha\in\mathcal{D}_n}\int d(x,\alpha)^rd\nu(x)\bigg)^{\frac{1}{r}}.
\end{eqnarray}
By \cite[Lemma 3.4]{GL:00}, the quantization error $e_{n,r}(\nu)$ is equal to the minimum error in approximation of $\nu$ with discrete probability measures which are supported on at most $n$ points in the $L_r$-metric.

If the infimum in (\ref{error}) is attained at some $\alpha\in\mathcal{D}_n$, we call $\alpha$ an $n$-optimal set for $\nu$ of order $r$. Let us call points of $\alpha$ \emph{$n$-optimal points} for $\nu$ of order $r$. By \cite[Theorem 4.12]{GL:00}, the collection $C_{n,r}(\nu)$ of all the $n$-optimal set for $\nu$ of order $r$ is non-empty whenever the $r$th moment $\int|x|^rd\nu(x)$ is finite.

The asymptotic properties for the $n$-th quantization error for $\nu$ of order $r$ have been deeply studied for absolutely continuous measures and some singular measures which are supported on fractals (cf. \cite{DT:10,GL:00,GL:03,LJL:01,MR:15,Kr:08,KZ:16,PK:01,Zhu:11}). Next, let us recall a significant concern in quantization theory.

Let $\alpha\subset\mathbb{R}^q$ be a finite set. A Voronoi partition (VP) with respect to $\alpha$ is a
Borel partition $\{P_a(\alpha):a\in\alpha\}$ of $\mathbb{R}^q$ which satisfies
\begin{eqnarray*}
P_a(\alpha)\subset\big\{x\in\mathbb{R}^q:\;{\rm d}(x,a)={\rm d}(x,\alpha)\big\}\;{\rm for\;all}\;\;a\in\alpha.
\end{eqnarray*}
We write $I_a(\alpha,\nu):=\int_{P_a(\alpha)}d(x,a)^rd\nu(x)$ and define
\begin{eqnarray*}
\underline{J}(\alpha,\nu):=\min\limits_{a\in\alpha}I_a(\alpha,\nu),\;\overline{J}(\alpha,\nu):=\max\limits_{a\in\alpha}I_a(\alpha,\nu).
\end{eqnarray*}
A famous conjecture of Gersho (cf. \cite{Ger:79,GL:12}) suggests that for $\alpha_n\in C_{n,r}(\nu)$ and  an arbitrary VP $\{P_a(\alpha)\}_{a\in\alpha}$ with respect to $\alpha_n$, the following holds:

\[
\underline{J}(\alpha_n,\nu),\overline{J}(\alpha_n,\nu)\sim\frac{1}{n}e^r_{n,r}(\nu).
\]
Here, $a_n\sim b_n$ means $a_n/b_n\to 1$ as $n\to\infty$. This conjecture is significant for all probability measures with finite $r$th moment. However, up to now, it has been proved true only for some special classes of one-dimensional probability distributions (cf. \cite{FG:02,GL:12,Kr:13}).

In 2012, Graf, Luschgy and Pag\`{e}s proved that, for a large class of absolutely continuous measures on $\mathbb{R}^q$, a weak version of Gersho's conjecture holds \cite{GL:12}:
\begin{equation}\label{GL}
\underline{J}(\alpha_n,\nu),\overline{J}(\alpha_n,\nu)\asymp\frac{1}{n}e^r_{n,r}(\nu),
\end{equation}
where $a_n\asymp b_n$ indicates that $Cb_n\leq a_n\leq C^{-1}b_n$ for all $n\geq 1$.
For general measures on $\mathbb{R}^q$, it is very difficult even to examine whether (\ref{GL}) holds or not. Therefore, it is significant to ask, for what measures (\ref{GL}) holds.

In the study of the above question, the following quantity for bounded Borel sets $A$ often plays a significant role:
\[
\mathcal{E}_r(A):=\nu(A)|A|^r,
\]
where $|A|$ denotes the diameter of the set $A$. Roughly speaking, we often expect that, for well-behaved probability measures (cf. Lemma \ref{l3}), the optimal points "should", in some sense, be distributed according to the size of $\mathcal{E}_r(A)$. With the above idea in mind, the author proved (\ref{GL}) for Ahlfors-David measures on $\mathbb{R}^q$ (see \cite{Zhu:20}). Recall that a Borel measure $\nu$ is called an $s$-dimensional Ahlfors-David measure if there exist constants $C,\epsilon_0>0$ such that
\[
C\epsilon^s\leq\nu(B(x,\epsilon))\leq C^{-1}\epsilon^s
\]
for every $x\in {\rm supp}(\nu)$ and $\epsilon\in (0,\epsilon_0)$. Here and hereafter, $B(x,\epsilon)$ denotes the closed ball of radius $\epsilon$ which is centered at a point $x\in\mathbb{R}^q$.

In \cite{Zhu:19}, the author proved that (\ref{GL}) is true for the Moran measures on $\mathbb{R}^1$. The Moran measures are the image measures of infinite product measures on the corresponding coding space under the natural projection. The advantage of these measures is, that an interval $I$ can always be excluded from its complement by its two endpoints, so that when we adjust the number of prospective optimal points in $I$, its complement would not be affected unfavorably. However, this is not applicable for Moran measures in higher-dimensional spaces. One of the major obstacles is that, for a given cylinder set $A$ (see Definition \ref{def1}), we are unable to estimate the number of the cylinder sets $B$, with $A,B$ non-overlapping and $\mathcal{E}_r(B)\asymp\mathcal{E}_r(A)$, whose $\epsilon$-neighborhoods intersect that of $A$, no matter how small $\epsilon$ is. Hence, a significant direction of effort is to seek some conditions, under which the above-mentioned numbers are bounded by some constant and then manage to apply the covering technique as descried in \cite{KZ:15} by Kesseb\"{o}hmer and Zhu.

In the present paper, we will prove that, (\ref{GL}) holds for the doubling measures on Moran sets in $\mathbb{R}^q$. We will assume a version of the open set condition which allows cylinder sets to touch one another.

Let $(n_k)_{k=1}^\infty$ be a sequence of positive integers with $\min\limits_{k\geq 1}n_k\geq 2$. For every $k\geq 1$, let $s_{k,j},1\leq j\leq n_k$, be real numbers in $(0,1)$ such that
\begin{eqnarray}\label{zhu1}
\sum_{j=1}^{n_k}s_{k,j}\leq 1;\;\;\inf_{k\geq 1}\min_{1\leq j\leq n_k}s_{k,j}=:\underline{s}>0.
\end{eqnarray}
We denote the empty word by $\theta$. We write
\begin{eqnarray*}
&&\Omega_k:=\prod_{j=1}^k\{1,\ldots,n_j\},\;\;\Psi_{k,h}:=\prod_{j=k+1}^{k+h}\{1,\ldots,n_j\},\;k,h\in\mathbb{N};\\
&&\Omega_{\mathbb{N}}:=\prod_{j=1}^\infty\{1,\ldots,n_j\};\;\;\Omega^*:=\bigcup_{k=1}^\infty\Omega_k.
\end{eqnarray*}

Let $\overline{A}, A^\circ$ denote the closure and interior in $\mathbb{R}^q$ of a set $A\subset\mathbb{R}^q$ respectively. For $k,h\geq 1$, $\sigma\in\Omega_k$ and $\omega\in\Psi_{k,h}$, we write $\sigma\ast\omega$ for the concatenation of $\sigma$ and $\omega$.
\begin{definition}\label{def1}
Let $J$ be a nonempty compact subset of $\mathbb{R}^q$ with $\overline{J^\circ}=J$.
Let $J_\theta:=J$. Let $J_i,i\in\Omega_1$, be subsets of $J$ such that
\begin{enumerate}
\item[(i)]the sets $J_i$ are geometrically similar to $J$ and $|J_i|/|J|=s_{1,i}$;
\item[(ii)] $J_i^\circ\cap J_k^\circ=\emptyset$ for every pair $1\leq i\neq k\leq n_1$.
\end{enumerate}
Let us call the sets $J_i$ \emph{cylinder sets of order one}. Assume that $J_\sigma,\sigma\in\Omega_k$, are defined. For each $\sigma\in\Omega_k$, let $J_{\sigma\ast i}, 1\leq i\leq n_{k+1}$, be subsets of $J_\sigma$ such that
\begin{enumerate}
\item[(1)]they are geometrically similar to $J_\sigma$ and $|J_{\sigma\ast i}|/|J_\sigma|=s_{k+1,i}$;
\item[(2)] $J_{\sigma\ast i}^\circ\cap J_{\sigma\ast k}^\circ=\emptyset$ for every pair $1\leq i\neq k\leq n_{k+1}$.
\end{enumerate}
Inductively, $J_\sigma$ is well defined for all $\sigma\in\Omega^*$. We call $J_\sigma,\sigma\in\Omega_k$, \emph{cylinder sets of order $k$}. We define
\[
E:=\bigcap_{k=1}^\infty\bigcup_{\sigma\in\Omega_k}J_\sigma.
\]
We call the set $E$ a \emph{Moran set} associated with $J, (n_k)_{k=1}^\infty$ and $\big((c_{k,j})_{j=1}^{n_k}\big)_{k\geq 1}$.
\end{definition}

Moran sets are important objects in fractal geometry. In the past decades, this type of sets and the measures supported on them have been of great interest to mathematicians (cf. \cite{CM:92,HRWW:00,LW:05,Moran:46,wen:00}).

Note that $(E,d)$ is a compact doubling metric space: there exists some integer $H_0\geq 1$ such that for every $\epsilon>0$ and every ball $B(x,2\epsilon)\cap E$ in the sub-metric space $(E,d)$ can be covered by at most $H_0$ balls of radii $\epsilon$ in $(E,d)$. This can be seen by considering a maximal family of pairwise disjoint balls of radii $2^{-1}\epsilon$ which are centered in $B(x,2\epsilon)\cap E$ and estimating the volumes. Therefore, by \cite{VK:87} (see also \cite{KRS:11,wjm:98}), $E$ carries a doubling measure---a Borel measure  $\mu$ such that, for some constant $D\geq 1$,
\begin{equation}\label{dm}
0<\mu(B(x,2\epsilon))\leq D\mu(B(x,\epsilon))<\infty\;\;{\rm for\;all}\;\;x\in E\;{\rm and}\;\epsilon>0.
\end{equation}

From (\ref{dm}), we know that $E$ is the topological support of $\mu$, and since $E$ is bounded and $E\subset B(x,|E|)$ for every $x\in E$, we also have that $\mu(E)<\infty$. Thus, $E$ always carries a doubling probability measure. Next, let us make some remarks on the doubling measures $\mu$ on $E$.

First, by Proposition 4.9 of \cite{Fal:04}, if $\mu$ is an $s$-dimensional Ahlfors-David measure with ${\rm supp}(\mu)=E$, then $E$ is an $s$-set, that is, the $s$-dimensional Hausdorff measure of $E$ is both positive and finite. However, according to Theorem 1.1 of \cite{HRWW:00}, a Moran set $E$ is not necessarily an $s$-set even if (\ref{zhu1}) is assumed. Thus, $E$ may not support an Ahlfors-David measure, but as we mentioned above, it always supports a doubling probability measure.

Secondly, let $f_i,1\leq i\leq N$, be contractive similitudes on $\mathbb{R}^q$. By \cite{Hut:81}, there exists a unique non-empty compact set which satisfies $F=\bigcup_{i=1}^Nf_i(F)$. The set $F$ is called the self-similar set associated with $(f_i)_{i=1}^N$. We say that $(f_i)_{i=1}^N$ satisfies the open set condition (OSC), if there exists a non-empty bounded open set $U$ such that $\bigcup_{i=1}^Nf_i(U)\subset U$ and
$f_i(U)\cap f_j(U)=\emptyset$ for every pair $1\leq i\neq j\leq N$. With the assumption of the OSC, $F$ is a Moran set as defined above (cf. \cite{GL:95}). Now let $(p_i)_{i=1}^N$ be a probability vector. There exists a unique Borel probability measure $\nu$ which satisfies $\nu=\sum_{i=1}^Np_i\nu\circ f_i^{-1}$. The measure $\nu$ is called the self-similar measure associated with $(f_i)_{i=1}^N$ and $(p_i)_{i=1}^N$.

In \cite{Young:07}, with the assumption of the OSC, Young established a necessary and sufficient condition for a self-similar measure to be doubling on $F$. By Proposition 1.5 of \cite{Young:07}, one can see that a doubling measure $\nu$ carried by $F$ needs not to be an Ahlfors-David measure, although it is well known that under the OSC, $F$ is an $s$-set and the normalized $s$-dimensional Hausdorff measure $H^s(\cdot|E)$ is an $s$-dimensional Ahlfors-David measure. One may also see \cite{Wen:15} for characterizations for the doubling measures carried by some Moran sets.

Further, if $(f_i)_{i=1}^N$ satisfies the strong separation condition, namely, $f_i(F)$, $1\leq i\leq N$, are pairwise disjoint, then by Olsen \cite{Olsen:95}, we know that all self-similar measures on $F$ are doubling.

Now we are able to state our main result.
Let $\partial A$ denote the boundary (in $\mathbb{R}^q$) of a set $A\subset\mathbb{R}^q$. We further assume that there exists some constants $\delta>0$ and $k_0\in\mathbb{N}$ such that, for every $\sigma\in\Omega^*$, there exists some $\tau(\sigma)\in\Psi_{|\sigma|,|\tau(\sigma)|}$ with $|\tau(\sigma)|\leq k_0$ which satisfies
 \begin{equation}\label{A1}
J_{\sigma\ast\tau(\sigma)}\subset J_\sigma^\circ\;{\rm and}\;\;d(J_{\sigma\ast\tau(\sigma)}, \partial J_\sigma)\geq\delta|J_\sigma|.
\end{equation}
When $E$ is a self-similar set, the condition (\ref{A1}) is guaranteed by the OSC (cf. Proposition 3.4 of \cite{GL:95}). This condition will enable us to estimate the $\mu$-measure of the boundary of $J_\sigma$ for every $\sigma\in\Omega^*$. By the assumption $\underline{s}>0$, (\ref{A1}) and the construction of $E$, it is not difficult to see that (cf. \cite{HRWW:00})
\begin{eqnarray}\label{san5}
\overline{s}:=\sup_{k\geq 1}\max_{1\leq j\leq n_k}s_{k,j}<1;\;\;N_0:=\sup_{k\geq 1} n_k<\infty.
\end{eqnarray}

As the main result of the present paper, we will prove that, (\ref{GL}) holds for the doubling measures on $E$. That is,
\begin{theorem}\label{mthm}
Let $E$ be a Moran set satisfying (\ref{A1}) and $\mu$ a doubling probability measure satisfying (\ref{dm}). For each $n\geq 1$, let  $\alpha_n$ be an arbitrary element of $C_{n,r}(\mu)$ and $\{P_a(\alpha_n)\}_{a\in\alpha_n}$ an arbitrary VP with respect to $\alpha_n$. Then
\[
\underline{J}(\alpha_n,\mu),\;\overline{J}(\alpha_n,\mu)\asymp\frac{1}{n}e_{n,r}^r(\mu).
\]
\end{theorem}

The remaining part of the paper is organized as follows.
In section 2, we will establish some basic facts for the measure $\mu$ and some auxiliary measures. Using these facts, we define, in section 3, some auxiliary integers. In section 4, we use these integers to establish estimates for the number of optimal points lying in the suitably chosen neighborhoods of cylinders, which may intersect one another. Finally, based on the estimates in section 4, we apply \cite[Theorem 4.1]{GL:00} and some results in \cite{Zhu:20} to complete the proof of the theorem.
\section{Preliminary lemmas}

Let $\sigma=(\sigma(1),\ldots,\sigma(k))\in\Omega_k$, we define $|\sigma|=:k$. For $1\leq h\leq k$, we write
\[
\sigma|_h:=(\sigma(1),\ldots,\sigma(h)).
\]
If $\sigma\in\Omega_1$, we define $\sigma^-=\theta$; if $|\sigma|>1$, we define $\sigma^-:=\sigma|_{|\sigma|-1}$. For $\sigma\in\Omega^*$ and $\tau\in\Omega^*\cup\Omega_{\mathbb{N}}$, we write $\sigma\prec\tau$ if $\sigma=\tau|_{|\sigma|}$. We say that $\sigma, \tau$ are incomparable if  we have neither $\sigma\prec\tau$ nor $\tau\prec\sigma$. By the construction of $E$, for every pair $\sigma,\tau$ of incomparable words, we have
$J_\sigma^\circ\cap J_\tau^\circ=\emptyset$.

A subset $\Gamma$ of $\Omega^*$ is called an antichain if the words in $\Gamma$ are pairwise incomparable; $\Gamma$ is called a maximal finite antichain if it is a finite antichain and for every $\rho\in\Omega_{\mathbb{N}}$, there exists some $\sigma\in\Gamma$ such that $\sigma\prec\rho$.
Without loss of generality, in the following, we assume that $|J|=1$. Then
\[
|J_\sigma|=s_\sigma:=\prod_{h=1}^{|\sigma|}s_{h,\sigma(h)},\;\sigma=(\sigma(1),\ldots,\sigma(|\sigma|)).
\]
For every $\tau\in\Psi_{k,h}$, we define $s_\tau$ in the same manner as we did for words in $\Omega^*$. Let $x\in\mathbb{R}$, we denote by $[x]$ the largest integer not exceeding $x$. Next, we will establish some basic properties for the measure $\mu$.

\begin{lemma}\label{l0000}
There exists a constant $D_0>0$ such that
\begin{equation}\label{sgz6}
\mu(J_{\sigma\ast\tau(\sigma)})\geq D_0\mu(J_\sigma)\;\;{\rm for\;every}\;\;\sigma\in\Omega^*.
\end{equation}
\end{lemma}
\begin{proof}
Let $\sigma\in\Omega^*$ be given.
Fix an arbitrary $y_0\in J_{\sigma\ast\tau(\sigma)\ast\tau(\sigma\ast\tau(\sigma))}\cap E$. By (\ref{A1}), we have
$d(y_0,\partial J_{\sigma\ast\tau(\sigma)})\geq \delta|J_{\sigma\ast\tau(\sigma)}|\geq\delta s_{\tau(\sigma)}s_\sigma$. It follows that
\[
B(y_0,2^{-1}\delta \underline{s}^{k_0}s_\sigma)\subset B(y_0,2^{-1}\delta s_{\tau(\sigma)}s_\sigma)\subset J_{\sigma\ast\tau(\sigma)}.
\]
Let $l_0:=\min\{k:2^{k-1}>\delta^{-1} \underline{s}^{-k_0}\}$ and $t_0:=2^{-1}\delta \underline{s}^{k_0}$. Then $2^{l_0}t_0>1$. Hence,
\[
B(y_0,t_0s_\sigma)\subset J_{\sigma\ast\tau(\sigma)}\subset J_\sigma\subset B(y_0,2^{l_0}t_0s_\sigma).
\]
Thus, by (\ref{dm}), we obtain
\begin{eqnarray}\label{z2}
\mu(J_\sigma)\leq D^{l_0}\mu(B(y_0,t_0s_\sigma))\leq D^{l_0}\mu(J_{\sigma\ast\tau(\sigma)}).
\end{eqnarray}
Hence, the lemma  follows by defining $D_0:=D^{-l_0}$.

\end{proof}
\begin{lemma}\label{l100}
For every $\sigma\in\Omega^*$, we have $\mu(\partial J_\sigma)=0$. As a consequence, we have $\mu(J_\sigma\cap J_\omega)=0$ for every pair $\sigma,\omega$ of incomparable words in $\Omega^*$.
\end{lemma}
\begin{proof}
Note that $J_{\sigma\ast\tau(\sigma)}\subset J_\sigma^\circ$. By Lemma \ref{l0000}, we obtain
\[
\mu(\partial J_\sigma)\leq\mu(J_\sigma)(1-D_0).
\]
Now for every $\omega=\Psi_{|\sigma|,|\tau(\sigma)|}\setminus\{\tau(\sigma)\}$, we apply (\ref{z2}) to $\sigma\ast\omega$ and get
\[
\mu(J_{\sigma\ast\omega\ast\tau(\sigma\ast\omega)})\geq D_0\mu(J_{\sigma\ast\omega}).
\]
Note that $J_{\sigma\ast\omega\ast\tau(\sigma\ast\omega)}\subset J_{\sigma\ast\omega}^\circ\subset J_\sigma^\circ$. We obtain
\[
\mu(\partial J_\sigma)\leq\mu(J_\sigma)(1-D_0)^2.
\]
By induction, we deduce that $\mu(\partial J_\sigma)\leq\mu(J_\sigma)(1-D_0)^k$ for all $k\geq 1$. Thus, we conclude that $\mu(\partial J_\sigma)=0$. For every pair of incomparable words $\sigma,\omega$, we know that $J_\sigma^\circ\cap J_\omega^\circ=\emptyset$. Hence, $\mu(J_\sigma\cap J_\omega)=0$.
\end{proof}
\begin{lemma}\label{l2}
There exists a number $p\in(0,1)$ such that
\begin{eqnarray}\label{z3}
p\mu(J_{\sigma})\leq\mu(J_{\sigma\ast i})\leq (1-p)\mu(J_\sigma).
\end{eqnarray}
for every $\sigma\in\Omega^*$ and $1\leq i\in n_{|\sigma|+1}$.
\end{lemma}
\begin{proof}
Let $\sigma\in\Omega^*$ and $1\leq i\leq n_{|\sigma|+1}$. Let $x_0$ be an arbitrary point of $J_{\sigma\ast i\ast\tau(\sigma\ast i)}\cap E$. Note that $|\tau(\sigma\ast i)|\leq k_0$. Using (\ref{A1}), we deduce
\[
d(x_0,\partial J_{\sigma\ast i})\geq\delta|J_{\sigma\ast i}|=\delta s_{\sigma\ast i}\geq\underline{s}\delta s_\sigma.
\]
We write $\xi :=\delta\underline{s}$. Then we have
\[
B(x_0,2^{-1}\xi s_\sigma)\subset J_{\sigma\ast i}\subset J_\sigma\subset B(x_0,s_\sigma).
\]
Let $k_1:=\min\{k\in\mathbb{N}:\xi ^{-1}<2^{k-1}\}$. Then by (\ref{dm}), we deduce
\begin{eqnarray}\label{z1}
\mu(J_\sigma)\leq\mu(B(x_0,s_\sigma))\leq D^{k_1}\mu(B(x_0,2^{-1}\xi s_\sigma))\leq D^{k_1}\mu(J_{\sigma\ast i}).
\end{eqnarray}
We define $p:=D^{-k_1}$. Then the first inequality in (\ref{z3}) is fulfilled. By our assumption, we have
$n_{|\sigma|+1}\geq 2$; thus, by Lemma \ref{l100}, we obtain the second inequality in (\ref{z3}). This completes the proof of the lemma.
\end{proof}

For every $\sigma\in\Omega^*$ and $\alpha\subset\mathbb{R}^q$, we write
\[
\mathcal{E}_r(\sigma):=\mu(J_\sigma)s_\sigma^r,\;\;I_\sigma(\alpha,\mu):=\int_{J_\sigma}d(x,\alpha)^rd\mu(x).
\]

Our next lemma connects the quantity $\mathcal{E}_r(\sigma)$  with some integrals over $J_\sigma$. It will be used to establish estimates for the quantization error for $\mu$.
\begin{lemma}\label{l3}
Let $H$ be an integer with $H\geq 2$. Let $\zeta>0,k_2:=[\frac{\log H}{\log 2}]+1$ and $C_{1,H}:=D_0p^{k_2}(\delta\underline{s}^{k_2})^r$. Let $\alpha$ be a subset of $(J_\sigma)_{\zeta|J_\sigma|}$ with ${\rm card}(\alpha)=H$.  Then
\begin{eqnarray*}
C_{1,H}\mathcal{E}_r(\sigma)\leq I_\sigma(\alpha,\mu)\leq (1+\zeta)^r\mathcal{E}_r(\sigma).
\end{eqnarray*}
\end{lemma}
\begin{proof}
By the hypothesis, for every $x\in J_\sigma$, we have
\[
d(x,\alpha)\leq (1+\zeta)|J_\sigma|=(1+\zeta)s_\sigma.
 \]
 It follows that $I_\sigma(\alpha,\mu)\leq(1+\zeta)^r\mathcal{E}_r(\sigma)$. It remains to give an estimate in the reverse direction. Note that
 ${\rm card}(\Psi_{|\sigma|,k_2})\geq 2^{k_2}>L$. There exists some $\omega\in\Phi_{|\sigma|,k_2}$ such that $J_{\sigma\ast\omega}^\circ\cap\alpha=\emptyset$. Hence, by (\ref{A1}), we obtain
\begin{equation}\label{z4}
\inf_{x\in J_{\sigma\ast\omega\ast\tau(\sigma\ast\omega})}d(x,\alpha)\geq\delta|J_{\sigma\ast\omega}|=\delta s_\omega|J_\sigma|\geq\delta\underline{s}^{k_2}s_\sigma.
\end{equation}
On the other hand, by Lemmas \ref{l0000} and \ref{l2}, we have
\begin{equation}\label{z5}
\mu(J_{\sigma\ast\omega\ast\tau(\sigma\ast\omega}))\geq D_0\mu(J_{\sigma\ast\omega})\geq D_0p^{k_2}\mu(J_\sigma).
\end{equation}
By using (\ref{z4}) and (\ref{z5}), we deduce
\begin{eqnarray*}
I_\sigma(\alpha,\mu)&\geq&I_{\sigma\ast\omega\ast\tau(\sigma\ast\omega)}(\alpha,\mu)\geq D_0p^{k_2}\mu(J_\sigma)(\delta\underline{s}^{k_2}s_\sigma)^r=C_{1,H}\mathcal{E}_r(\sigma).
\end{eqnarray*}
This completes the proof of the lemma.
\end{proof}

 Let $\eta_r:=\min\{p\underline{s}^r,8^{-r}\}$. By Lemma \ref{l2}, we know that
 \[
 \mathcal{E}_r(\sigma)\leq (1-p)\overline{s}^r\mathcal{E}_r(\sigma^-)<\mathcal{E}_r(\sigma^-).
 \]
 This allows us to define the following finite maximal antichain in $\Omega^*$:
\begin{eqnarray}\label{lambdakr}
\Lambda_{k,r}:=\big\{\sigma\in\Omega^*:\mathcal{E}_r(\sigma^-)\geq\eta_r^k>\mathcal{E}_r(\sigma)\big\};\;\phi_{k,r}:={\rm card(\Lambda_{k,r})}.
\end{eqnarray}
\begin{remark}\label{rem02}
We have $\phi_{k,r}\leq\phi_{k+1,r}\leq N_1\phi_{k,r}$, where
\[
H_0:=\bigg[\frac{\log\eta_r}{\log(1-p)\overline{s}^r}\bigg]+1,\;\;N_1:=N_0^{H_0}.
\]
This can be seen as follows. For every $\sigma\in\Lambda_{k,r}$, by Lemma \ref{l2} and (\ref{lambdakr}),
 \[
 \eta_r^k>\mathcal{E}_r(\sigma)\geq\mathcal{E}_r(\sigma^-)\eta_r=\eta_r^{k+1}.
 \]
 Note that $n_k\geq 2$ for all $k\geq 1$. We deduce that $\phi_{k,r}\leq\phi_{k+1,r}$. For every $\sigma\in\Lambda_{k,r}$ and $\omega\in\Psi_{|\sigma|,H_0}$, again, by Lemma \ref{l2} and (\ref{lambdakr}), we have
\begin{eqnarray*}\label{san6}
\mathcal{E}_r(\sigma\ast\omega)\leq\mathcal{E}_r(\sigma)\big((1-p)\overline{s}^r\big)^{H_0}<\eta_r^{k+1}.
\end{eqnarray*}
This and (\ref{san5}) implies that $\phi_{k+1,r}\leq N_0^{H_0}=N_1\phi_{k,r}$.
\end{remark}

For a set $F\subset\mathbb{R}^q$ and $\zeta>0$, we write $(F)_\zeta$ for the closed $\zeta$-neighborhood of $F$. For every $\sigma\in\Lambda_{k,r}$, we define
\begin{eqnarray}\label{sgzhu1}
\mathcal{A}_\sigma:=\big\{\omega\in\Lambda_{k,r}:(J_\sigma)_{\frac{s_\sigma}{4}}\cap (J_\omega)_{\frac{s_\omega}{4}}\neq\emptyset\big\}, \;M_\sigma={\rm card}(\mathcal{A}_\sigma).
\end{eqnarray}
One can see that $\omega\in\mathcal{A_\sigma}$ if and only if $\sigma\in\mathcal{A_\omega}$.

 \begin{lemma}\label{l4}
There exists constants $C_2$ and $C_3$ such that for every pair $\sigma,\omega\in\Lambda_{k,r}$ with $\omega\in\mathcal{A}_\sigma$, we have
\begin{equation}\label{zhusanguo2}
C_2s_\sigma\leq s_\omega<C_2^{-1}s_\sigma;\;\;C_3\mu(J_\omega)\leq\mu(J_\sigma)\leq C_3^{-1}\mu(J_\omega).
\end{equation}
\end{lemma}
\begin{proof}
Let $\sigma\in\Lambda_{k,r}$ and $\omega\in\mathcal{A}_\sigma$.  It suffices to show that there exists a constant $C$ such that whenever $s_\omega<\frac{1}{8}s_\sigma$, we have $s_\omega\geq Cs_\sigma$. Assume that $s_\omega<\frac{1}{8}s_\sigma$. Let $x_0$ be an arbitrary point in $J_{\sigma\ast\tau(\sigma)}\cap E$. Then by (\ref{A1}), we have
\[
B(x_0,\frac{1}{2}\delta s_\sigma)\subset J_\sigma;\;J_\sigma\cup J_\omega\subset B(x_0,\frac{45}{32}s_\sigma).
\]
Let $k_3:=\min\{k:2^k>45/(16\delta)\}$. By (\ref{dm}), we deduce
\begin{eqnarray*}
\mu(J_\omega)\leq\mu\big(B(x_0,\frac{45}{32}s_\sigma)\big)\leq D^{k_3}\mu(B(x_0,\frac{1}{2}\delta s_\sigma))\leq D^{k_3}\mu(J_\sigma).
\end{eqnarray*}
Now by (\ref{lambdakr}), we know that $\mathcal{E}_r(\omega)\geq\eta_r\mathcal{E}_r(\sigma)$. It follows that
\[
\frac{s_\omega^r}{s_\sigma^r}\geq\eta_r\frac{\mu(J_\sigma)}{\mu(J_\omega)}\geq\frac{\eta_r}{D^{k_3}}.
\]
The first part of (\ref{zhusanguo2}) follows by defining $C_2:=(\eta_r/D^{k_3})^{1/r}$. To see the second, we define $C_3:=\eta_r C_2^r$. Then by the first part of (\ref{zhusanguo2}) and (\ref{lambdakr}), we have
\[
\frac{\mu(J_\sigma)}{\mu(J_\omega)}\geq\eta_r\frac{s_\omega^r}{s_\sigma^r}\geq\eta_rC_2^r.
\]
This completes the proof of the lemma.
\end{proof}

With the above preparations, we are now able to establish an upper bound for the numbers $M_\sigma,\sigma\in\Lambda_{k,r}$ as defined in (\ref{sgzhu1}).
\begin{lemma}\label{l6}
 There exists a constant $M_0$ such that $\max\limits_{\sigma\in\Lambda_{k,r}}M_\sigma\leq M_0$.
\end{lemma}
\begin{proof}
For every $\omega\in\mathcal{A}_\sigma$, we fix an arbitrary $x_\omega\in J_{\omega\ast\tau(\omega)}\cap E$ and an arbitrary $x_\sigma\in J_\sigma\cap E$. By (\ref{sgzhu1}) and Lemma \ref{l4}, we have
\begin{eqnarray*}
B(x_\omega,2^{-1}C_2\delta s_\sigma)\subset B(x_\omega,2^{-1}\delta s_\omega)\subset J_\omega^\circ\subset J_\omega\subset B(x_\sigma,(\frac{5}{4}(1+C_2^{-1}))s_\sigma).
\end{eqnarray*}
Since the words in $\mathcal{A}_\sigma$ are pairwise incomparable, the balls $B(x_\omega,2^{-1}C_2\delta s_\sigma)$, $\omega\in\mathcal{A}_\sigma$, are mutually disjoint. Hence, by estimating the volumes, we obtain
\begin{eqnarray*}
M_\sigma(2^{-1}C_2\delta s_\sigma)^q\leq\big(\frac{5}{4}(1+C_2^{-1})s_\sigma\big)^q.
\end{eqnarray*}
By defining $M_0:=(5(1+C_2^{-1}))^q(2C_2\delta)^{-q}$, the lemma follows.
\end{proof}
\begin{remark}
The boundedness of the set $\{M_\sigma:\sigma\in\Lambda_{k,r},k\geq 1\}$ will be very crucial for us to establish a characterization for the optimal sets. Unfortunately, without the doubling property, we are unable to obtain this boundedness even for self-similar measures with the assumption of the OSC.
\end{remark}

Next, we define some auxiliary measures which are image measures of the conditional measures of $\mu$ on cylinder sets $J_\sigma$. On one hand, these auxiliary measures will allow us to extract the crucial quantity $\mathcal{E}_r(\sigma)$; on the other hand, as we will see, they share some basic properties which will be very helpful for the characterizations for the optimal sets.

For every $\sigma\in\Omega^*$, let $g_\sigma$ be an arbitrary similitude with similarity ratio $s_\sigma$. Let $\mu(\cdot|J_\sigma)$ denote the conditional probability measure of $\mu$ on $J_\sigma$. We define
\begin{equation}\label{sgzhu2}
\nu_\sigma:=\mu(\cdot|J_\sigma)\circ g_\sigma,\;K_\sigma:={\rm supp}(\nu_\sigma).
\end{equation}
Then one can see that $K_\sigma\subset g_\sigma^{-1}(J_\sigma)$ and $|K_\sigma|\leq 1$. We have
\begin{lemma}\label{lem01}
There exist constants $C_4$ and $t$ such that, for every $\sigma\in\Omega^*$ and $\epsilon>0$, we have $\sup\limits_{x\in\mathbb{R}^q}\nu_\sigma(B(x,\epsilon))\leq C_4\epsilon^t$.
\end{lemma}
\begin{proof}
Let $t:=\frac{\log(1-p)}{\log\underline{s}}$. By Lemmas \ref{l2}, (\ref{sgzhu2}) and (\ref{zhu1}), we have
\begin{eqnarray}\label{zsg1}
&&p\nu_\sigma(g_\sigma^{-1}(J_{\sigma\ast\tau^-}))\leq\nu_\sigma(g_\sigma^{-1}(J_{\sigma\ast\tau}))\leq (1-p)\nu_\sigma(g_\sigma^{-1}(J_{\sigma\ast\tau^-}));\\
&&\underline{s}|g_\sigma^{-1}(J_{\sigma\ast\tau^-})|\leq|g_\sigma^{-1}(J_{\sigma\ast\tau})|\leq \overline{s}|g_\sigma^{-1}(J_{\sigma\ast\tau^-})|.\label{zsg2}
\end{eqnarray}
The lemma can be proved by using (\ref{zsg1}), (\ref{zsg2}) and the same argument as that in the proof for \cite[Proposition 5.1]{GL:04}.
\end{proof}
\begin{remark}\label{zhusanguo01}
Let $L_0:=49^q$. Then $(J_\sigma)_{\frac{s_\sigma}{4}}$ can be covered by $L_0$ closed balls of radii $\frac{s_\sigma}{16}$ which are centered in $(J_\sigma)_{\frac{|s_\sigma}{4}}$. In fact, we may consider the largest number $L$ of closed balls of radii $\frac{s_\sigma}{32}$ which are centered in $(J_\sigma)_{\frac{|s_\sigma}{4}}$, and then double the radii and obtain a cover for $(J_\sigma)_{\frac{|s_\sigma}{4}}$. By estimating the volumes, one can see that $L\leq L_0$.
\end{remark}
\begin{remark}\label{zhusanguo1}
Let $\alpha\subset\mathbb{R}^q$ be a nonempty finite set and $\sigma\in\Lambda_{k,r}$. Let $B_\sigma$ denote the set of the centers of some $L_0$ closed balls of radii $\frac{s_\sigma}{16}$ which are centered in $(J_\sigma)_{\frac{s_\sigma}{4}}$ and cover $(J_\sigma)_{\frac{s_\sigma}{4}}$. For $\omega\in\Lambda_{k,r}\setminus\mathcal{A}_\sigma$, we have $(J_\omega)_{\frac{s_\omega}{4}}\cap(J_\sigma)_{\frac{s_\sigma}{4}}=\emptyset$. Thus, by triangle inequality, one can see that  for every $x\in J_\omega$, we have
\[
d(x,(\alpha\setminus(J_\sigma)_{\frac{s_\sigma}{8}})\cup B_\sigma))\leq d(x,\alpha).
\]
Thus, if we replace $\alpha$ with $(\alpha\setminus(J_\sigma)_{\frac{s_\sigma}{8}})\cup B_\sigma)$, only the points in $\bigcup_{\omega\in\mathcal{A}_\sigma}J_\omega$ might be affected unfavorably.
\end{remark}

For every $\sigma\in\Lambda_{k,r}$, let $\mathcal{A}_\sigma$ be as defined in (\ref{sgzhu1}) and $\alpha\subset\mathbb{R}^q$. Motivated by Remark \ref{zhusanguo1}, we define
\[
A_\sigma^*:=\bigcup_{\omega\in\mathcal{A}_\sigma}J_\omega;\;\;I_\sigma^*(\alpha,\mu):=\int_{A_\sigma^*}d(x,\alpha)^rd\mu(x);\;
\mathcal{E}_r^*(\sigma):=\mu(A_\sigma^*)|A_\sigma^*|^r.
\]
\begin{lemma}\label{san1}
There exists a constant $D_1$ such that, for every $\sigma\in\Lambda_{k,r}$, the following holds:
\[
D_1\mathcal{E}_r(\sigma)\leq\mathcal{E}_r^*(\sigma)\leq D_1^{-1}\mathcal{E}_r(\sigma).
\]
\end{lemma}
\begin{proof}
By Lemmas \ref{l4}-\ref{l6}, we have
\[
|A_\sigma^*|\leq \frac{5}{2}(1+C_2^{-1})s_\sigma,\;\;\mu(A_\sigma^*)\leq M_0C_3^{-1}\mu(J_\sigma)
\]
We define $D_1:=M_0^{-1}C_3(\frac{5}{2}(1+C_2^{-1}))^{-r}$. The lemma follows.
\end{proof}
Let $h_\sigma$ be an arbitrary similitude with similarity ratio $|A_\sigma^*|$. We define
\begin{equation}\label{zhu3}
\nu_\sigma^*:=\mu(\cdot|A_\sigma^*)\circ h_\sigma,\;\;K_\sigma^*:={\rm supp}(\nu_\sigma^*).
\end{equation}
\begin{lemma}\label{lemma1}
There exists a constant $C_5$ such that, for every $\sigma\in\Lambda_{k,r}$ and $\epsilon>0$, we have $\sup_{x\in\mathbb{R^q}}\nu_\sigma^*(B(x,\epsilon))\leq C_5\epsilon^t$.
\end{lemma}
\begin{proof}
Let $x\in \mathbb{R}^q$ and $\epsilon>0$. Using (\ref{sgzhu2}), (\ref{zhu3}) and Lemma \ref{l100}, we deduce
\begin{eqnarray}\label{zhu4}
\nu_\sigma^*(B(x,\epsilon))&=&\frac{1}{\mu(A_\sigma^*)}\mu(B(h_\sigma(x),|A_\sigma^*|\epsilon)\cap A_\sigma^*))\nonumber\\
&=&\frac{\sum_{\tau\in\mathcal{A}_\sigma}\mu(B(h_\sigma(x),|A_\sigma^*|\epsilon)\cap J_\tau)}{\sum_{\tau\in\mathcal{A}_\sigma}\mu(J_\tau)}\nonumber\\
&\leq&\max_{\tau\in\mathcal{A}_\sigma}\frac{\mu(B(h_\sigma(x),|A_\sigma^*|\epsilon)\cap J_\tau)}{\mu(J_\tau)}\nonumber\\
&=&\max_{\tau\in\mathcal{A}_\sigma}\nu_\tau\circ g_\tau^{-1}(B(h_\sigma(x),|A_\sigma^*|\epsilon))\nonumber\\
&=&\max_{\tau\in\mathcal{A}_\sigma}\nu_\tau(B(g_\tau^{-1}\circ h_\sigma(x),s_\tau^{-1}|A_\sigma^*|\epsilon)).
\end{eqnarray}
By Lemma \ref{l4} and (\ref{sgzhu1}), for every $\tau\in\mathcal{A}_\sigma$, we have
 \begin{equation}\label{zhu5}
 |A_\sigma^*|/s_\tau\leq 4^{-1}5(1+C_2^{-1}+C_2^{-2})=\widetilde{C}_4.
 \end{equation}
Let $C_5:=C_4\widetilde{C}_4^t$. By (\ref{zhu4}), (\ref{zhu5}) and Lemma \ref{lem01}, we obtain
\[
\nu_\sigma^*(B(x,\epsilon))\leq C_4\widetilde{C}_4^t\epsilon^t=C_5\epsilon^t.
\]
This completes the proof of the lemma.
\end{proof}

\section{Auxiliary integers}
First, we select three integers $M_i,1\leq i\leq 3$, which will be used to establish a lower bound for the number of optimal points lying in $(J_\sigma)_{\frac{s_\sigma}{8}}$. The following Lemmas \ref{l7}-\ref{l9} are devoted to this goal.

\begin{lemma}\label{l7}
Let $\alpha\subset\mathbb{R}^q$ and $\sigma\in\Lambda_{k,r}$. Assume that there exists some point $x_0$ in  $J_\sigma\cap E$ such that $d(x_0,\alpha)>8^{-1}s_\sigma$. Then for $k_4=[-\frac{\log(16)}{\log\overline{s}}]+1$ and $C_6:=p^{k_4}(16)^{-r}$,  we have
\[
I_\sigma(\alpha,\mu)\geq C_6\mathcal{E}_r(\sigma).
\]
\end{lemma}
\begin{proof}
By the hypothesis, for every $y\in B(x_0,\frac{1}{16}s_\sigma)$, we have
$d(y,\alpha)\geq \frac{1}{16}s_\sigma$. Note that for every $\tau\in\Psi_{|\sigma|,k_4}$, we have
\[
|J_{\sigma\ast\tau}|\leq\overline{s}^{k_4}s_\sigma<\frac{1}{16}s_\sigma.
\]
Therefore, there exists some $\tau_1\in\Phi_{|\sigma|,k_4}$ such that
$x_0\in J_{\sigma\ast\tau_1}\subset B(x_0,\frac{1}{16}s_\sigma)$.
Hence, for every $y\in J_{\sigma\ast\tau_1}$, we have $d(y,\alpha)\geq \frac{1}{16}s_\sigma$. It follows that
\[
d(g_\sigma^{-1}(J_{\sigma\ast\tau_1}),g_\sigma^{-1}(\alpha))\geq \frac{1}{16}.
\]
Using this and Lemma \ref{l2}, we deduce
\begin{eqnarray*}
I_\sigma(\alpha,\mu)&=&\mathcal{E}_r(\sigma)\int d(x,g_\sigma^{-1}(\alpha))^rd\nu_\sigma(x)\\&\geq&\mathcal{E}_r(\sigma)\int_{g_\sigma^{-1}(J_{\sigma\ast\tau_1})} d(x,g_\sigma^{-1}(\alpha))^rd\nu_\sigma(x)\\&\geq& p^{k_4}(16)^{-r}\mathcal{E}_r(\sigma).
\end{eqnarray*}
This completes the proof of the lemma.
\end{proof}

The next lemma is an easy consequence of the definition of the quantization errors and some covering techniques.

\begin{lemma}\label{lem02}(see \cite[Lemma 2.2]{Zhu:20})
Let $\nu$ be a Borel probability measure on $\mathbb{R}^q$ with compact support $K_\nu$. Then for every $\eta>0$, there exists an integer $M(\eta)$ depending only on $\eta$ and $q$, such that $l\geq M(\eta)$ implies
\[
e^r_{l,r}(\mu)\leq(\eta|K_\nu|)^r.
\]
\end{lemma}

\begin{lemma}\label{l8}
There exists a smallest integer $M_1$ such that, for every $\sigma\in\Lambda_{k,r}$,
\[
e^r_{l,r}(\nu_\sigma), \;e^r_{l,r}(\nu_\sigma^*)<D_1C_6\;\;{\rm for\; every}\; l\geq M_1.
\]
In particular, for every $\gamma\in C_{l,r}(\nu_\sigma^*)$, we have
\[
I_\sigma^*(h_\sigma(\gamma),\mu)<C_6\mathcal{E}_r(\sigma).
\]
\end{lemma}
\begin{proof}
Let $\eta:=(2^{-1}C_6D_1)^{1/r}$ and $M_1:=M(\eta)$. Note that $|K_\sigma|,|K_\sigma^*|\leq 1$. By Lemma \ref{lem02}, for $l\geq M_1$, we obtain
\[
e^r_{l,r}(\nu_\sigma),\;e^r_{l,r}(\nu_\sigma^*)\leq(\eta|K_\nu|)^r\leq2^{-1}C_6D_1<C_6D_1.
\]
Now let $\gamma\in C_{l,r}(\nu_\sigma^*)$. By Lemma \ref{san1}, we have
\begin{eqnarray*}
I_\sigma^*(h_\sigma(\gamma),\mu)&=&\mathcal{E}_r^*(\sigma)\int d(x,\gamma)^rd\nu_\sigma^*(x)=\mathcal{E}_r^*(\sigma)e^r_{l,r}(\nu_\sigma^*)\\&<&D_1C_6\mathcal{E}_r^*(\sigma)\leq C_6\mathcal{E}_r(\sigma).
\end{eqnarray*}
This completes the proof of the lemma.
\end{proof}

Let $L_0$ be as defined in Remark \ref{zhusanguo01} and $M_2:=M_1+L_0$. We have
\begin{lemma}\label{l00}
Let $\alpha\subset\mathbb{R}^q$ and $\sigma\in\Lambda_{k,r}$. If ${\rm card}(\alpha\cap(J_\sigma)_{\frac{s_\sigma}{8}})<M_2$, then
\[
I_\sigma(\alpha,\mu)\geq \mathcal{E}_r(\sigma)e_{M_2-1,r}^r(\nu_\sigma)).
\]
\end{lemma}
\begin{proof}
Let $\sigma\in\Lambda_{k,r}$ and $\alpha\subset\mathbb{R}^q$. We write
\begin{equation}\label{gsigma}
G_\sigma:=\{x\in J_\sigma\cap E: d(x,\alpha)=d(x,\alpha\setminus(J_\sigma)_{\frac{s_\sigma}{8}})\}.
\end{equation}
We have the following two cases:

Case (a1): $G_\sigma=\emptyset$. In this case, we have $d(x,\alpha)=d(x,\alpha\cap (J_\sigma)_{\frac{s_\sigma}{8}})$ for every $x\in J_\sigma\cap E$.
Note that $K_\sigma\subset g_\sigma^{-1}(E\cap J_\sigma)$. It follows that
\begin{eqnarray*}
I_\sigma(\alpha,\mu)&=&\mu(J_\sigma)\int_{J_\sigma}d(x,\alpha)^rd\nu_\sigma \circ g_\sigma^{-1}(x)\\&=&\mathcal{E}_r(\sigma)\int d(x,g_\sigma^{-1}(\alpha\cap (J_\sigma)_{\frac{s_\sigma}{8}}))^rd\nu_\sigma(x)\\&\geq&\mathcal{E}_r(\sigma) e^r_{M_2-1,r}(\nu_\sigma)).
\end{eqnarray*}

Case (a2) $G_\sigma\neq\emptyset$. Then by Lemmas \ref{l7} and \ref{l8}, we obtain
\begin{eqnarray*}
I_\sigma(\alpha,\mu)\geq C_6\mathcal{E}_r(\sigma)>e^r_{M_1,r}(\nu_\sigma)\mathcal{E}_r(\sigma)\geq e^r_{M_2-1,r}(\nu_\sigma)\mathcal{E}_r(\sigma).
\end{eqnarray*}
The lemma follows by combining the above analysis.
\end{proof}

\begin{lemma}\label{lem03}(see \cite[Lemma 2.3]{Zhu:20})
Let $\nu$ be a Borel probability measure on $\mathbb{R}^q$ with support $K_\nu$. Assume that $|K_\nu|\leq 1$ and there exist constants $C,s>0$ such that $\sup_{x\in K_\nu}\mu(x,\epsilon)\leq C\epsilon^s$. Then there exists a $\zeta_{l,r}$ depending only on $l,r,C,s$ and $q$ such that
$e^r_{l-1,r}(\nu)-e^r_{l,r}(\nu)\geq\zeta_{l,r}$.
\end{lemma}
\begin{lemma}\label{l9}
There exists an integer $M_3>M_2+L_0$ such that for  $l\geq M_3$ and every pair $\sigma,\omega\in\Lambda_{k,r}$, the following holds:
\[
e^r_{l-M_2-L_0}(\nu_\sigma^*)<D_1\eta_r(e^r_{M_2-1}(\nu_\omega)-e^r_{M_2}(\nu_\omega)).
\]
In particular, for every $l\geq M_3$ and $\gamma\in C_{l-M_2-L_0,r}(\nu_\sigma^*)$, we have
\end{lemma}
\[
I_\sigma^*(h_\sigma(\gamma),\mu)<\mathcal{E}_r(\omega)\big(e^r_{M_2-1}(\nu_\omega)-e^r_{M_2}(\nu_\omega)\big).
\]
\begin{proof}
Note that $|K_\sigma^*|\leq 1$. We set
\[
\eta:=(2^{-1}D_1\eta_r\zeta_{M_2,r})^{1/r},\;M_3:=M(\eta)+M_2+L_0.
\]
Then by Lemmas \ref{lem02} and \ref{lem03}, for all $l\geq M_3$, we obtain
\begin{eqnarray*}
e^r_{l-M_2-L_0}(\nu_\sigma^*)\leq2^{-1}D_1\eta_r\zeta_{M_2,r}<D_1\eta_r(e^r_{M_2-1}(\nu_\omega)-e^r_{M_2}(\nu_\omega)).
\end{eqnarray*}
Let $l\geq M_3$ and $\gamma\in C_{l-M_2-L_0,r}(\nu_\sigma^*)$. Using Lemma \ref{san1} and (\ref{lambdakr}), we deduce
\begin{eqnarray*}
I_\sigma^*(h_\sigma(\gamma),\mu)&=&\mathcal{E}_r^*(\sigma)e^r_{l,r}(\nu_\sigma^*)\\
&<&\mathcal{E}_r^*(\sigma)D_1\eta_r(e^r_{M_2-1}(\nu_\omega)-e^r_{M_2}(\nu_\omega))\\
&\leq&\mathcal{E}_r(\sigma)\eta_r(e^r_{M_2-1}(\nu_\omega)-e^r_{M_2}(\nu_\omega))
\\&\leq&\mathcal{E}_r(\omega)(e^r_{M_2-1}(\nu_\omega)-e^r_{M_2}(\nu_\omega)).
\end{eqnarray*}
This completes the proof of the lemma.
\end{proof}

\begin{remark}
Let $N_1$ be as defined in Remark \ref{rem02}. We define
\[
M_4:=M_3+L_0\;\;{\rm and}\;\;M_5:=N_1M_4.
\]
For every $n\geq M_4\phi_{1,r}$, there exists a unique $k\in\mathbb{N}$, such that
\begin{equation}\label{n+k}
M_4\phi_{k,r}\leq n<M_4\phi_{k+1,r}.
\end{equation}
By Remark \ref{rem02}, we know that $\phi_{k+1,r}\leq N_1\phi_{k,r}$. Thus, we have
\[
M_4\phi_{k,r}\leq n<N_1M_4\phi_{k,r}=M_5\phi_{k,r}.
\]
\end{remark}

In the following, we will use Lemmas \ref{lem05}-\ref{lem04} to select three more integers $M_i, 5\leq i\leq 7$. These integers will be used to establish an upper bound for the numbers of $n$-optimal points lying in $(J_\sigma)_{\frac{s_\sigma}{8}},\sigma\in\Lambda_{k,r}$.
\begin{lemma}\label{lem05}
Let $H\geq 1$ be an integer. Then there exists a constant $\xi_{H,r}$ which depends on $r, C_4, t$ and $H$, such that, for every $\sigma\in\Omega^*$, we have
\[
e_{H,r}^r(\nu_\sigma)\geq \xi_{H,r}.
\]
\end{lemma}
\begin{proof}
Let $\gamma\in C_{H,r}(\nu_\sigma)$. Let $\epsilon_H:=\frac{1}{(4HC_4)^{1/t}}$. By Lemma \ref{lem01}, we have
\begin{equation}\label{sgzhu3}
\nu_\sigma(\bigcup_{b\in\gamma}B(b,\epsilon_H))\leq\sum_{b\in\gamma}\nu_\sigma(B(b,\epsilon_H))\leq \frac{1}{4}.
\end{equation}
As a consequence of (\ref{error}) and (\ref{sgzhu3}), we obtain
\begin{eqnarray*}
e_{H,r}^r(\nu_\sigma)\geq\int_{K_\sigma\setminus\bigcup_{b\in\gamma}B(b,\epsilon_H)}d(x,\gamma)^rd\nu_\sigma(x)
\geq\frac{3}{4}\epsilon_H^r.
\end{eqnarray*}
By defining $\xi_{H,r}:=\frac{3}{4}\epsilon_H^r$, the proof of the lemma is complete.
\end{proof}
\begin{lemma}\label{san3}
Let $\emptyset\neq\alpha\subset\mathbb{R}^q$. There exists a constant $C_7>0$ such that for every $\sigma\in\Lambda_{k,r}$ with ${\rm card}(\alpha\cap J_\sigma^\circ)<M_5$, we have
\[
I_\sigma(\alpha,\mu)\geq C_7\mathcal{E}_r(\sigma).
\]
\end{lemma}
\begin{proof}
By the assumption (\ref{A1}), we have $d(J_{\sigma\ast\tau(\sigma)},\partial J_\sigma)\geq \delta s_\sigma$. We write
\[
S(\sigma):=\big\{x\in J_{\sigma\ast\tau(\sigma)}\cap E:d(x,\alpha)=d(x,\alpha\setminus J^\circ)\big\}.
\]
We distinguish between the following two cases.

Case (b1): $S(\sigma)=\emptyset$. In this case, we have
\begin{eqnarray}\label{sz1}
I_{\sigma\ast\tau(\sigma)}(\alpha,\mu)&=&\mathcal{E}_r(\sigma\ast\tau(\sigma))\int d(x,g_\tau^{-1}(\alpha\cap J_\sigma^\circ))d\nu_{\sigma\ast\tau(\sigma)}(x)
\nonumber\\&\geq&\mathcal{E}_r(\sigma\ast\tau(\sigma))e^r_{M_5-1,r}(\nu_{\sigma\ast\tau(\sigma)}).
\end{eqnarray}
By Lemma \ref{lem05}, we have
$e^r_{M_5-1,r}(\nu_{\sigma\ast\tau(\sigma)})\geq \xi_{M_5-1}$. This and (\ref{sz1}) yield
\begin{eqnarray}\label{sz3}
I_\sigma(\alpha,\mu)\geq I_{\sigma\ast\tau(\sigma)}(\alpha,\mu)\geq \xi_{M_5-1}\mathcal{E}_r(\sigma\ast\tau(\sigma))\geq \xi_{M_5-1}p^{k_0}\underline{s}^{k_0}\mathcal{E}_r(\sigma).
\end{eqnarray}

Case (b2): $S(\sigma)\neq\emptyset$. Fix an arbitrary point $x_0\in S(\sigma)$. By (\ref{A1}), we have
\[
d(x_0,\alpha)\geq d(x_0,\partial J_\sigma)\geq\delta s_\sigma.
 \]
Thus, we have $d\big(B(x_0,4^{-1}\delta s_\sigma),\alpha\big)>4^{-1}\delta s_\sigma$. Let $k_5:=[\frac{\log(\delta/4)}{\log\overline{s}}]+1$. Then for every $\rho\in\Psi_{|\sigma|,k_5}$, we have
$|J_{\sigma\ast\rho}|\leq s_\sigma\overline{s}^{k_5}<\frac{\delta}{4}s_\sigma$.
Therefore, there exists some $\rho\in\Phi_{|\sigma|,|\rho|}$ with $|\rho|\leq k_5$ such that
\[
x_0\in J_{\sigma\ast\rho},\;\;d(J_{\sigma\ast\rho},\alpha)\geq 4^{-1}\delta s_\sigma.
\]
Using this and Lemma \ref{l2}, we deduce
\begin{eqnarray}\label{sz2}
I_\sigma(\alpha,\mu)\geq I_{\sigma\ast\rho}(\alpha_n,\mu)\geq p^{k_5}(4^{-1}\delta)^r\mathcal{E}_r(\sigma).
\end{eqnarray}
Combining (\ref{sz3}) and (\ref{sz2}), the lemma follows by defining
\[
C_7:=\min\big\{\xi_{M_5-1}p^{k_0}\underline{s}^{k_0},p^{k_5}(4^{-1}\delta)^r\big\}.
\]
\end{proof}
\begin{lemma}\label{san4}
There exists a constant $M_6$ such that, $e^r_{l,r}(\nu_\sigma)<\frac{1}{2}C_7$ for every $l\geq M_6$ and $\sigma\in\Omega^*$.
In particular, for every $l\geq M_6$ and $\gamma\in C_{l,r}(\nu_\sigma)$, we have
\[
I_\sigma(g_\sigma(\gamma),\mu)<\frac{1}{2}C_7\mathcal{E}_r(\sigma).
\]
\end{lemma}
\begin{proof}
By Lemma \ref{lem02}, it suffices to define $\eta:=(4^{-1}C_7)^{1/r}$ and $M_6:=M(\eta)$.
\end{proof}

\begin{lemma}\label{lem04}
There exists a smallest integer $M_7>M_0M_5+M_6+L_0$ such that for  $l\geq M_7$ and $\sigma\in\Lambda_{k,r}$, the following holds:
\[
e^r_{l-M_6-L_0}(\nu_\sigma^*)<2^{-1}D_1\eta_rC_7.
\]
In particular, for every $\omega\in\Lambda_{k,r}\setminus\{\sigma\}$ and $\gamma\in C_{l-M_6-L_0}(\nu_\sigma^*)$, we have
\[
I_\sigma^*(h_\sigma(\gamma),\mu)<2^{-1}C_7\mathcal{E}_r(\omega).
\]
\end{lemma}
\begin{proof}
By Lemmas \ref{lem02}, \ref{san1} and (\ref{lambdakr}), it suffices to define
 \[
 \eta:=(4^{-1}D_1\eta_rC_7)^{1/r}\;\; {\rm and}\;\;M_7:=M(\eta)+M_0M_5+M_6+L_0.
 \]
\end{proof}

\section{A characterization of the $n$-optimal sets}

We always assume that $\alpha_n\in C_{n,r}(\mu)$ and $k$ satisfies (\ref{n+k}).
We denote by $B_\sigma$ the set of the centers of some $L_0$ balls of radii $\frac{1}{16}s_\sigma$ which are centered in $(J_\sigma)_{\frac{s_\sigma}{4}}$ and cover $(J_\sigma)_{\frac{s_\sigma}{4}}$.  We define
\[
\kappa_\sigma:={\rm card}(\alpha_n\cap (J_\sigma)_{\frac{s_\sigma}{8}}),\;\sigma\in\Lambda_{k,r}.
\]

In the following, we will use three lemmas to establish upper and lower estimates for the numbers $\kappa_\sigma,\sigma\in\Lambda_{k,r}$.
The first lemma  can be proved by using the argument in the proof for \cite[Lemma 3.1]{Zhu:20}.

\begin{lemma}\label{l0}
We have $\kappa_c:={\rm card}\big(\alpha_n\setminus\big(\bigcup_{\sigma\in\Lambda_{k,r}}(J_\sigma)_{\frac{s_\sigma}{8}}\big)\big)\leq L_0\phi_{k,r}$.
\end{lemma}

Using Lemmas \ref{l0} and \ref{l8}-\ref{l9}, we are able to give a lower bound for $\kappa_\sigma$ for all $\sigma\in\Lambda_{k,r}$. That is,
\begin{lemma}\label{l1}
For every $\sigma\in\Lambda_{k,r}$, we have $\kappa_\sigma\geq M_2$.
\end{lemma}
\begin{proof}
Assume that $\kappa_\sigma<M_2$ for some $\sigma\in\Lambda_{k,r}$. By (\ref{n+k}) and Lemma \ref{l0},
\begin{eqnarray*}
\sum_{\tau\in\Lambda_{k,r}\setminus\{\sigma\}}\kappa_\tau&\geq&n-{\rm card}(\alpha_n\setminus\bigcup_{\sigma\in\Lambda_{k,r}}(J_\sigma)_{\frac{s_\sigma}{8}})-M_2\\&\geq&(M_4-L_0)\phi_{k,r}-M_2>M_3(\phi_{k,r}-1).
\end{eqnarray*}
Therefore, there exists some $\tau\in\Lambda_{k,r}\setminus\{\sigma\}$ such that $\kappa_\tau>M_3$. Let
\begin{eqnarray*}
&&\gamma_{\kappa_\tau-L_0-M_2}(\tau)\in C_{\kappa_\tau-L_0-M_2,r}(\nu_\tau^*);\;\gamma_{M_2}(\sigma)\in C_{M_2,r}(\nu_\sigma);\\
&&\beta:=\big(\alpha_n\setminus(J_\tau)_{\frac{s_\tau}{8}}\big)\cup B_\tau\cup h_\tau(\gamma_{\kappa_\tau-L_0-M_2}(\tau))\cup g_\sigma(\gamma_{M_2}(\sigma)).
\end{eqnarray*}
 Then we have ${\rm card}(\gamma)\leq n$. By Remark \ref{zhusanguo1}, we obtain,
\begin{eqnarray}\label{s01}
\sum_{\omega\in\Lambda_{k,r}\setminus(\mathcal{A}_\tau\cup\{\sigma\})}I_\omega(\beta,\mu)\leq
\sum_{\omega\in\Lambda_{k,r}\setminus(\mathcal{A}_\tau\cup\{\sigma\})}I_\omega(\alpha_n,\mu).
\end{eqnarray}
In the following, we distinguish between two cases.

Case (c1): $\sigma\notin\mathcal{A}_\tau$. In this case, we have $(J_\sigma)_{\frac{s_\sigma}{4}}\cap(J_\tau)_{\frac{s_\tau}{4}}=\emptyset$.
Note that $\kappa_\tau>M_3$. By Lemmas \ref{l00}, \ref{l9} and (\ref{lambdakr}), we deduce
\begin{eqnarray}
&&I_\sigma(\alpha_n,\mu)-I_\sigma(\beta,\mu)\nonumber\\&&\geq
I_\sigma(\alpha_n,\mu)-\int_{J_\sigma}d(x,g_\sigma(\gamma_{M_2}(\sigma))^rd\mu(x)\nonumber
\\&&\geq\mathcal{E}_r(\sigma)(e^r_{M_2-1}(\nu_\sigma)-e^r_{M_2}(\nu_\sigma))\nonumber\\&&>I_\tau^*(h_\tau(\gamma_{\kappa_\tau-L_0-M_2}(\tau)),\mu)
\label{sgzhu5}\\&&\geq I_\tau^*(\beta,\mu)\nonumber\\&&\geq I_\tau^*(\beta,\mu)-I_\tau^*(\alpha_n,\mu).\label{zhu2}
\end{eqnarray}
From (\ref{s01})-(\ref{zhu2}), we obtain that $I(\beta,\mu)<I(\alpha_n,\mu)$, a contradiction.

Case (c2): $\sigma\in\mathcal{A}_\tau$. In this case, $J_\sigma\subset A_\tau^*$.
Using (\ref{sgzhu5}) and (\ref{zhu2}), we deduce
\begin{eqnarray*}
I_\tau^*(\beta,\mu)&=&\sum_{\omega\in\mathcal{A}_\tau\setminus\{\sigma\}}I_\omega(\beta,\mu)+I_\sigma(\beta,\mu)\\&\leq&I_\tau^*(\beta,\mu)+I_\sigma(\beta,\mu)
\\&\leq&I_\tau^*(h_\tau(\gamma_{\kappa_\tau-L_0-M_2}(\tau),\mu)+I_\sigma(\beta,\mu)
\\&<&(I_\sigma(\alpha_n,\mu)-I_\sigma(\beta,\mu))
+I_\sigma(\beta,\mu)<I_\tau^*(\alpha_n,\mu).
\end{eqnarray*}
Combining this and (\ref{s01}), we deduce that $I(\beta,\mu)<I(\alpha_n,\mu)$, a contradiction.
\end{proof}

Next, by Lemmas \ref{san4}-\ref{lem04}, we establish an upper bound for $\kappa_\sigma$ for all $\sigma\in\Lambda_{k,r}$. This will be used to establish a lower bound for $\underline{J}(\alpha_n,\mu)$.

\begin{lemma}\label{lem06}
For every $\sigma\in\Lambda_{k,r}$, we have $\kappa_\sigma\leq M_7$.
\end{lemma}
\begin{proof}
Assume that $\kappa_\sigma>M_7$ for some $\sigma\in\Lambda_{k,r}$. Next, we will deduce a contradiction. Note that $M_7>M_0M_5$ and $n<M_5\phi_{k,r}$. Further, for every $\tau\in\Lambda_{k,r}\setminus\mathcal{A}_\sigma$, we have $(J_\tau)_{\frac{s_\tau}{4}}\cap(J_\sigma)_{\frac{s_\sigma}{4}}=\emptyset$. Thus,
\[
{\rm card}\bigg(\alpha_n\cap\bigcup_{\tau\in\Lambda_{k,r}\setminus\mathcal{A}_\sigma}(J_\tau)_{\frac{s_\tau}{8}}\bigg)<(\phi_{k,r}-M_0)M_5.
\]
Since ${\rm card}(\mathcal{A}_\sigma)\leq M_0$ (Lemma \ref{l6}), we obtain
${\rm card}(\Lambda_{k,r}\setminus\mathcal{A}_\sigma)\geq \phi_{k,r}-M_0$.
Further, for distinct words $\tau,\rho\in\Lambda_{k,r}$, we have $J_\tau^\circ\cap J_\rho^\circ=\emptyset$. Thus, there exists some $\tau\in\Lambda_{k,r}\setminus\mathcal{A}_\sigma$ such that
${\rm card}(\alpha\cap J_\tau^\circ)<M_5$. Let $\gamma_{M_6}(\tau)\in C_{M_6,r}(\nu_\tau)$ and
\begin{eqnarray*}
\beta:=\big(\alpha_n\setminus(J_\sigma)_{\frac{s_\sigma}{8}}\big)\cup B_\sigma\cup h_\sigma(\gamma_{\kappa_\sigma-L_0-M_6}(\sigma))\cup g_\sigma(\gamma_{M_6}(\tau)).
\end{eqnarray*}
Then ${\rm card}(\beta)\leq n$. Again, by Remark \ref{zhusanguo1}, we have
\begin{eqnarray}\label{sanguo1}
\sum_{\omega\in\Lambda_{k,r}\setminus (\mathcal{A}_\sigma\cup \{\tau\})}I(\beta,\mu)\leq\sum_{\omega\in\Lambda_{k,r}\setminus (\mathcal{A}_\sigma\cup \{\tau\})}I(\alpha_n,\mu).
\end{eqnarray}
By Lemmas \ref{san3}-\ref{lem04}, we deduce
\begin{eqnarray*}
I_\tau(\alpha_n,\mu)-I_\tau(\beta,\mu)&\geq& I_\tau(\alpha_n,\mu)-I_\tau(g_\sigma(\gamma_{M_6}(\tau)),\mu)> 2^{-1}C_7\mathcal{E}_r(\tau)\\&>&
I_\sigma^*(h_\sigma(\gamma_{\kappa_\sigma-L_0-M_6}(\sigma)),\mu)\geq I_\sigma^*(\beta,\mu)\\&\geq& I_\sigma^*(\beta,\mu)-I_\sigma^*(\alpha_n,\mu).
\end{eqnarray*}
Using this and (\ref{sanguo1}), we obtain $I(\beta,\mu)<I(\alpha_n,\mu)$, a contradiction.
\end{proof}

Next we give an estimate for the distance between $\alpha_n$ and an arbitrary point in $J_\sigma\cap E$. The integer $M_1$ is defined mainly for this purpose.
\begin{lemma}\label{l12}
For every $\sigma\in\Lambda_{k,r}$, we have
$\sup\limits_{x\in J_\sigma\cap E}d(x,\alpha_n)\leq\frac{s_\sigma}{8}$. In particular,
\[
\alpha_n\subset\bigcup_{\sigma\in\Lambda_{k,r}}(J_\sigma)_{\frac{s_\sigma}{8}};\;\;d(x,\alpha_n)=d(x,\alpha_n\cap(J_\sigma)_{\frac{s_\sigma}{8}})\;{\rm for}\;\;x\in J_\sigma\cap E.
\]
\end{lemma}
\begin{proof}
Assume that, $d(x,\alpha_n)>8^{-1}s_\sigma$ for some $\sigma\in\Lambda_{k,r}$ and $x\in J_\sigma\cap E$. Next, we deuce a contradiction. By the assumption and Lemma \ref{l7}, we have
\begin{eqnarray}\label{s1}
I_\sigma^*(\alpha_n,\mu)\geq I_\sigma(\alpha_n,\mu)\geq C_6\mathcal{E}_r(\sigma).
\end{eqnarray}
Let $\gamma_{\kappa_\sigma-L_0}\in C_{\kappa_\sigma-L_0,r}(\nu_\sigma^*)$. We define a set $\beta$ with ${\rm card}(\beta)\leq n$:
\[
\beta:=\big(\alpha_n\setminus (J_\sigma)_{\frac{s_\sigma}{8}}\big)\cup B_\sigma\cup h_\sigma(\gamma_{\kappa_\sigma-L_0}).
 \]
By Remark \ref{zhusanguo1}, we obtain
\begin{eqnarray}\label{s2}
\sum_{\tau\in\Lambda_{k,r}\setminus\mathcal{A}_\sigma}I_\tau(\beta,\mu)\leq
\sum_{\tau\in\Lambda_{k,r}\setminus\mathcal{A}_\sigma}I_\tau(\alpha_n,\mu).
\end{eqnarray}
By Lemma \ref{l1}, we have, $\kappa_\sigma-L_0\geq  M_1$. Thus, by Lemma \ref{l8}, we obtain
\begin{eqnarray}\label{s3}
I_\sigma^*(\beta,\mu)<C_6\mathcal{E}_r(\sigma)\leq I_\sigma^*(\alpha_n,\mu).
\end{eqnarray}
Combining (\ref{s1})-(\ref{s2}), we obtain that $I(\alpha_n,\mu)>I(\beta,\mu)$. This contradicts the optimality of $\alpha_n$ and the lemma follows.
\end{proof}

\section{Proof of the main result}

As in section 4, we assume that $\alpha_n\in C_{n,r}(\mu)$, and $k$ satisfies (\ref{n+k}). Let $\{P_a(\alpha_n)\}_{a\in\alpha_n}$ be a VP with respect to $\alpha_n$ The following lemma gives a characterization for the geometric structure of the elements of  $\{P_a(\alpha_n)\}_{a\in\alpha_n}$.

\begin{lemma}\label{prop1}
For every $a\in\alpha_n$ and $\sigma\in\Lambda_{k,r}$, we have
\begin{eqnarray}
&&{\rm card}(\{\sigma\in\Lambda_{k,r}: P_a(\alpha)\cap J_\sigma\cap E\neq\emptyset\})\leq M_0;\label{sz6}\\
&&{\rm card}(\{a\in\alpha_n: P_a(\alpha)\cap J_\sigma\cap E\neq\emptyset\})\leq M_7.\label{sz7}
\end{eqnarray}
\end{lemma}
\begin{proof}
Let $a$ be an arbitrary point of $\alpha_n$. Then by Lemma \ref{l12}, there exists some $\sigma\in\Lambda_{k,r}$ such that $a\in(J_\sigma)_{\frac{s_\sigma}{8}}$. Note that for every $\omega\in\Lambda_{k,r}\setminus\mathcal{A}_\sigma$, we have $(J_\omega)_{\frac{s_\omega}{4}}\cap(J_\sigma)_{\frac{s_\sigma}{4}}=\emptyset$.
 Hence, for every $\omega\in\Lambda_{k,r}\setminus\mathcal{A}_\sigma$ and every $x\in J_\omega$, we have $d(x,a)>\frac{s_\omega}{8}$. By Lemma \ref{l12},
we obtain that, $P_a(\alpha_n)\cap E\subset A_\sigma^*$ and (\ref{sz6}) follows.  (\ref{sz7}) is an easy consequence of Lemmas \ref{lem06} and \ref{l12}.
\end{proof}

Using Lemma \ref{prop1} and \cite[Theorem 4.1]{GL:00}, we are able to reduce the quantization problem with respect to an arbitrarily large $n$ to that with respect to some bounded numbers. We need to consider the union of some bounded number of elements of $\{P_a(\alpha_n)\}_{a\in\alpha_n}$. Let $a$ be an arbitrary point in $\alpha_n$. Then $a\in(J_\sigma)_{\frac{s_\sigma}{8}}$ for some $\sigma\in\Lambda_{k,r}$. We define
\[
\Gamma(a):=\alpha_n\cap\bigcup_{\omega\in\mathcal{A}_\sigma}(J_\omega)_{\frac{s_\omega}{8}};\;\;T_a:={\rm card}(\Gamma(a)).
\]
By Lemma \ref{lem06}, $T_a\leq M_0M_7=:M_8$. By Lemma \ref{prop1}, we obtain
\begin{equation}\label{union}
H(a):=\bigcup_{b\in\Gamma(a)}P_b(\alpha_n)\cap E\subset\bigcup_{\omega\in\mathcal{A}_\sigma}A_\omega^*.
\end{equation}
Let $\sigma\in\Lambda_{k,r}$ and $\omega_i,1\leq i\leq k_6(\leq M_0)$ be an enumeration of $\mathcal{A}_\sigma$. For every $1\leq i\leq k_6$, let $F_i$ be a subset of $A_{\omega_i}^*\setminus J_{\omega_i}$. We write
\begin{eqnarray}\label{sgz3}
R_\sigma^*:=A_\sigma^*\cup\bigg(\bigcup_{i=1}^{k_6}F_i\bigg).
\end{eqnarray}
One can see that $H(a)=R_\sigma^*$ for some $k_6$ and some choice of $(F_i)_{1\leq i\leq k_6}$ and
\begin{equation}\label{temp1}
P_a(\alpha_n)\cap E\subset H(a);\;\;P_a(\alpha_n)\cap E=P_a(\Gamma(a))\cap E.
\end{equation}

In order to obtain a lower estimate for $\underline{J}(\alpha_n,\mu)$, we need to consider the conditional measure of $\mu$ on $R_\sigma^*$ and apply  \cite[Lemma 2.4]{Zhu:20}.
For this reason, we select an arbitrary similitude $f_\sigma$ of similarity ratio $|R_\sigma^*|$ and define
\begin{equation}\label{sgz4}
\lambda_\sigma^*:=\mu(\cdot|R_\sigma^*)\circ f_\sigma.
\end{equation}
By (\ref{sgzhu1}) and Lemma \ref{l4}, one may find a constant $C_8$ such that
\begin{eqnarray}\label{sanguo2}
\max_{1\leq i\leq k_6}|R_\sigma^*|A_{\omega_i}|^{-1}\leq\max_{1\leq i\leq k_6}|R_\sigma^*|s_{\omega_i}|^{-1}\leq C_8.
\end{eqnarray}

\begin{lemma}\label{lem08}
Let $t$ be as given in Lemma \ref{lem01}. There exists a constant $C_9>0$ such that, for every $\sigma\in\Lambda_{k,r}$ and every $\epsilon>0$, we have
\[
\sup_{x\in\mathbb{R}^q}\lambda_\sigma^*(B(x,\epsilon))\leq C_9\epsilon^t.
\]
\end{lemma}
\begin{proof}
Let $\epsilon>0$ and $x\in\mathbb{R}^q$. By (\ref{sgz4}) and Lemma \ref{l100}, we have
\begin{eqnarray}\label{sz90}
&&\lambda_\sigma^*(B(x,\epsilon))=\frac{1}{\mu(R_\sigma^*)}\mu(B(f_\sigma(x),|R_\sigma^*|\epsilon)\cap R_\sigma^*)\nonumber\\
&&\;\;\;\;=\frac{\sum_{i=1}^{k_6}\mu(B(f_\sigma(x),|R_\sigma^*|\epsilon)\cap J_{\omega_i})
+\sum_{i=1}^{k_6}\mu(B(f_\sigma(x),|R_\sigma^*|\epsilon)\cap F_i)}{\sum_{i=1}^{k_6}\mu(J_{\omega_i})+\sum_{i=1}^{k_6}\mu(F_i)}\nonumber\\
&&\;\;\;\;\leq\sum_{i=1}^{k_6}\frac{\mu(B(f_\sigma(x),|R_\sigma^*|\epsilon)\cap J_{\omega_i})
}{\mu(J_{\omega_i})}+\sum_{i=1}^{k_6}\frac{\mu(B(f_\sigma(x),|R_\sigma^*|\epsilon)\cap F_i)
}{\mu(J_{\omega_i})}.
\end{eqnarray}
Thus, by (\ref{sz90}) and  Lemmas \ref{l4} and \ref{l6}, we further deduce
\begin{eqnarray}\label{sz9}
\lambda_\sigma^*(B(x,\epsilon))&\leq&\sum_{i=1}^{k_6}\nu_{\omega_i}(B((g_{\omega_i}^{-1}\circ f_\sigma(x),s_{\omega_i}^{-1}|R_\sigma^*|\epsilon))\nonumber\\&&
+\frac{M_0}{C_3}\sum_{i=1}^{k_6}\nu_{\omega_i}^*(B(h_{\omega_i}^{-1}\circ f_\sigma(x),|A_{\omega_i}^*|^{-1}|R_\sigma^*|\epsilon)\cap A_{\omega_i}^*).
\end{eqnarray}
Note that $C_4\leq C_5$ and $k_6\leq M_0$. By (\ref{sanguo2}), (\ref{sz9}), Lemmas \ref{lem01}, \ref{lemma1},  we obtain
\[
\lambda_\sigma^*(B(x,\epsilon))\leq 2M_0^2C_5C_3^{-1}C_8^t\epsilon^t.
\]
It is sufficient to define $C_9:=2M_0^2C_5C_3^{-1}C_8^t$.
\end{proof}

For two $\mathbb{R}$-valued variables $X,Y$, we write $X\lesssim Y$ ($X\gtrsim Y$) if there exists some constant $T$ such that $X\leq TY$ ($X\geq TY$). Our next lemma provides us with estimates for $e_{n,r}(\mu)$ in terms of $\mathcal{E}_r(\sigma),\sigma\in\Lambda_{k,r}$.
\begin{lemma}\label{lem07}
We have $e^r_{n,r}(\mu)\asymp\sum_{\sigma\in\Lambda_{k,r}}\mathcal{E}_r(\sigma)\asymp\phi_{k,r}\eta_r^k$.
\end{lemma}
\begin{proof}
By Lemmas \ref{l100} and \ref{l12} and (\ref{lambdakr}), we obtain
\begin{eqnarray}\label{sgz1}
I(\alpha_n,\mu)=\sum_{\sigma\in\Lambda_{k,r}}I_\sigma(\alpha_n,\mu)\leq\sum_{\sigma\in\Lambda_{k,r}}\mathcal{E}_r(\sigma)\leq \phi_{k,r}\eta_r^k.
\end{eqnarray}
By applying Lemma \ref{l3} with $H:=M_7$, we obtain
\begin{equation}\label{sgzhu4}
I(\alpha_n,\mu)=\sum_{\sigma\in\Lambda_{k,r}}I_\sigma(\alpha_n\cap (J_\sigma)_{\frac{s_\sigma}{8}},\mu)\gtrsim\sum_{\sigma\in\Lambda_{k,r}}\mathcal{E}_r(\sigma)\gtrsim\phi_{k,r}\eta_r^k.
\end{equation}
Combining (\ref{sgz1}) and (\ref{sgzhu4}), the lemma follows.
\end{proof}

\begin{lemma}(cf. \cite[Lemma 2.4]{Zhu:20})\label{microapp}
Let $\nu$ be a Borel probability measure on $\mathbb{R}^q$ with compact support $K_\nu$ such that
$\sup_{x\in \mathbb{R}^q}\nu(B(x,\epsilon))\leq C\epsilon^{t}$ for every $\epsilon\in(0,\infty)$. Assume that $|K_\nu|\leq1$. Then for every $n\geq 1$, there exists a number $d_n>0$ which depends on $n,C,q,t$, such that
 \begin{eqnarray*}
 \inf_{\alpha_n\in C_{n,r}(\nu)}\underline{J}(\alpha_n,\nu)>d_n.
 \end{eqnarray*}
 \end{lemma}
\emph{Proof of Theorem \ref{mthm}}
Let $a$ be an arbitrary point of $\alpha_n$. By Lemma \ref{prop1}, we have $P_a(\alpha_n)\cap E\subset A_\sigma^*$ for some $\sigma\in\Lambda_{k,r}$. Thus, by Lemmas \ref{l12} and \ref{lem07},
\begin{eqnarray}\label{sgz5}
I_a(\alpha_n,\mu)\leq\sum_{\omega\in\mathcal{A}_\sigma}8^{-r}\mathcal{E}_r(\omega)\leq M_08^{-r}\eta_r^k\lesssim \frac{1}{n}e_{n,r}^r(\mu).
\end{eqnarray}
Let $H(a)$ be as defined in (\ref{union}). Then $H(a)=R_\sigma^*$ (see (\ref{sgz3})) for some $1\leq k_6\leq M_0$ and some $F_i,1\leq i\leq k_6$. By \cite[Theorem 4.1]{GL:00}, we have
\[
\Gamma(a)\in C_{T_a,r}(\mu(\cdot|H(a)))=C_{T_a,r}(\mu(\cdot|R_\sigma^*)).
\]
From the similarity of $f_\sigma$, we deduce that $f_\sigma^{-1}(\Gamma(a))\in C_{T_a,r}(\lambda_\sigma^*)$.
Using (\ref{temp1}), Lemmas \ref{lem08}, \ref{microapp} and the similarity of $f_\sigma$, we deduce
\begin{eqnarray*}
I_a(\alpha_n,\mu)&=&\int_{P_{a}(\Gamma(a))}d(x,a)^rd\mu(x)\\&=& \mu(R_\sigma^*)|R_\sigma^*|^r\int_{P_{f_\sigma^{-1}(a)}(f_\sigma^{-1}(\Gamma(a)))}d(x,f_\sigma^{-1}(a))^rd\lambda_\sigma^*(x)\\&\geq& \mu(R_\sigma^*)|R_\sigma^*|^r\min_{1\leq i\leq M_9}d_i\\&\gtrsim&\mathcal{E}_r(\sigma)\asymp\frac{1}{n}e_{n,r}^r(\mu).
\end{eqnarray*}
This and (\ref{sgz5}) complete the proof of Theorem \ref{mthm}.

\end{document}